\documentclass[12pt]{article}
\usepackage[]{amsmath,amssymb,amsfonts,latexsym,amsthm,enumerate,hyperref,eucal,verbatim}
\usepackage{pstricks, pst-node,pst-plot}
\psset{arrows=->, labelsep=3pt, mnode=circle}

\newtheorem{thm}{Theorem}
\numberwithin{thm}{section}
\newtheorem*{thm*}{Theorem}
\newtheorem{lemma}[thm]{Lemma}
\newtheorem*{lemma*}{Lemma}
\newtheorem{prop}[thm]{Proposition}
\newtheorem*{prop*}{Proposition}

\newtheorem{cor*}{Corollary}
\newtheorem*{rmk*}{Remark}

\newtheorem*{rmks*}{Remarks}
\newtheorem*{not*}{Notation}
\newtheorem*{claim*}{Claim}
\newtheorem*{fact*}{Fact}

\theoremstyle{definition}

\newtheorem*{namedprop}{\theoremname}
\newcommand{\theoremname}{testing}

\numberwithin{equation}{section}

\def\R{\mathbb{R}}
\def\C{\mathbb{C}}

\def\N{\mathbb{N}}
\def\Z{\mathbb{Z}}

\def\tr{\mathrm{tr \,}}
\def\LC{\mathrm{LC}}
\def\sign{\mathrm{sign \,}}
\def\6{\partial}

\begin{document}

\title{On the Measure of the Absolutely Continuous Spectrum for Jacobi Matrices}

\author{Mira Shamis\textsuperscript{1} and Sasha Sodin\textsuperscript{2}}
\footnotetext[1]{shamis@math.huji.ac.il, Einstein Institute of Mathematics
Edmond J. Safra Campus, Givat Ram, The Hebrew University of Jerusalem
Jerusalem 91904, Israel}
\footnotetext[2]{sodinale@post.tau.ac.il, School of Mathematical Sciences,
Tel Aviv University, Ramat Aviv, Tel Aviv 69978, Israel} \maketitle

\abstract{We apply the methods of classical approximation theory (extreme
properties of polynomials) to study the essential support $\Sigma_\text{ac}$
of the absolutely continuous spectrum of Jacobi matrices. First, we prove an upper bound
on the measure of $\Sigma_\text{ac}$ which takes into account the value distribution
of the diagonal elements, and implies the bound due to Deift--Simon and Poltoratski--Remling.

Second, we generalise the differential inequality of Deift--Simon for the
integrated density of states associated with the absolutely continuous
spectrum to general Jacobi matrices.}

\section{Introduction}

In this work we consider Jacobi matrices
\begin{equation}\label{eq:op}
J(a, b) =
    \left(\begin{array}{cccccc}
    b(1) & a(1) & 0 & 0 & 0 &\cdots \\
    a(1) & b(2) & a(2) & 0 & 0 &\cdots \\
    0 & a(2) & b(3) & a(3) & 0 & \cdots \\
    0 & 0 & \ddots & \ddots & \ddots &
    \end{array}\right)~.
\end{equation}
We assume that
\[ a(n) \in [c, C], b(n) \in [-C, C] \quad \text{for some $0 < c < C < \infty$.}\]

In this case $J(a, b)$ defines a bounded self-adjoint operator on $\ell^2(\N)$:
\begin{equation}\label{eq:op1}
(J \psi)(n) = a(n-1) \psi(n-1) + a(n) \psi(n+1) + b(n) \psi(n)~, \quad n \in \N~,
\end{equation}
where formally $a(0)=1$, $\psi(0)=0$.

The important subclass of {\em ergodic} operators is constructed as follows.
Let $(\Omega, \mu, T)$ be an ergodic system, and let $A, B: \Omega \to \R$
be bounded measurable functions. Then every $\omega \in \Omega$ defines an
operator $J_\omega = J(a_\omega, b_\omega)$, via
\[ a_\omega(n) = A(T^n \omega), \quad b_\omega(n) = B(T^n \omega)~. \]

If $J(a, b)$ is an arbitrary Jacobi matrix, we denote its spectrum
\[ \left\{ E \, \big| \, \text{$J-E$ is not invertible} \right\} \]
by $\sigma(J)$. The equivalence class of the set
\[ \Sigma_\text{ac} = \left\{ E \, \big| \,
    \text{$\lim_{\epsilon \downarrow 0} \mathrm{Im} (J - E - i \epsilon)^{-1}(1, 1)$ exists
    and differs from zero}\right\} \]
modulo sets of measure zero is called the essential support of the absolutely continuous spectrum.

For $n = 1, 2, \cdots$, denote
\[ k_n(E) = \frac{1}{n} \# \left\{ 1 \leq j \leq n \, \Big| \, \lambda_j \leq E \right\}~, \]
where $\lambda_j$ are the eigenvalues of the top-left $n \times n$ block of $J(a, b)$.
If the limit
\[ k(E) = \lim_{n \to \infty} k_n(E)\]
exists for almost every $E \in \R$, it is called the integrated density of states.

In particular, if $\{ J_\omega \}$ is an ergodic family of operators, it is known (see
Pastur and Figotin \cite{FP}) that the integrated density of states exists for almost
every $\omega \in \Omega$.

\vspace{2mm}

In 1983, Deift and Simon \cite{DS} proved  the following inequality.
If $J$ is an ergodic discrete Schr\"odinger operator, then for a.e.\
$E \in \Sigma_\text{ac}(J)$

\begin{equation}\label{eq:ds.loc}
\frac{d}{dE} \left\{ - 2 \cos \left( \pi k(E) \right) \right\}
    = 2 \pi \sin \left( \pi k(E) \right) \frac{dk(E)}{dE} \geq 1~.
\end{equation}

As $k(-\infty) = 0$ and $k(+\infty) = 1$, they immediately deduced

\begin{equation}\label{eq:ds.glob}
\left| \Sigma_\text{ac}(J) \right| \leq 4~.
\end{equation}

\noindent A different proof of (\ref{eq:ds.glob}) was given by Last
in \cite{L}, who deduced it from a stronger inequality.

In 2008, Poltoratski and Remling \cite{PR} proved that the inequality (\ref{eq:ds.glob}) holds for general
discrete Schr\"odinger operators (without assuming ergodicity). Moreover, the measure of the
essential support of the a.c.\ spectrum of a general Jacobi matrix $J(a, b)$ satisfies
\begin{equation}\label{eq:rp.glob}
\left| \Sigma_\text{ac}(J) \right| \leq 4 \liminf_{n \to \infty} A_n~,
\end{equation}
where

\begin{equation}
A_n = \Big[ a(1) a(2) \cdots a(n) \Big]^{\frac{1}{n}}~.
\end{equation}

Our goal is two-fold. First, we show that the inequality (\ref{eq:rp.glob}) follows from an extremal
property of Chebyshev polynomials (the P\'olya inequality).  This approach can be stretched further, to give
a more precise estimate on $\left| \Sigma_\text{ac}(J) \right|$ in terms of the value distribution of
the sequence $\{b(n)\}_{n=1}^\infty$.
For example, we prove an estimate on the measure of $\Sigma_\text{ac}$ intersected with an arbitrary
interval; the bound is always as good as the right-hand side of (\ref{eq:rp.glob}), and improves
on it when many of the $b(j)$ are far from our interval. To formulate the precise statement,
we introduce some notation.

Let $(E_L, E_R)$ be an interval (which may be infinite.) Consider the set of indices
\[ I_n = \left\{ 1 \leq j \leq n \, \big| \, b(j) \in (E_L - 2 M - A_n, E_R + 2M + A_n) \right\}~,\]
where
\[ M = \sup_{1 \leq j < \infty} a(j) = \frac{\|J(a, 0)\|}{2}~, \]
and let
\[ D_n = \prod_{1 \leq j \leq n, j \notin I_n}
                \Big( \min (|b(j)-E_R|, |b(j)-E_L|) - 2M\Big)~. \]

\begin{thm}\label{thm:1.1}
In the notation above,
\[ \left| \Sigma_\text{ac}(J(a, b)) \cap (E_L, E_R) \right|
    \leq 4 \liminf_{n \to \infty} \left[ \frac{A_n^n}{D_n} \right]^{\frac{1}{\# I_n}} \]
\end{thm}

\pagebreak
\begin{rmks*}\hfill
\begin{enumerate}
\item Note that every one of the factors in the definition of $D_n$ is at least $A_n$, since the product
is taken over
\[ b(j) \notin (E_L - 2 M - A_n, E_R + 2M + A_n)~, \]
and is much larger if $b(j)$ is far from $(E_L, E_R)$. Thus, $D_n$ measures the number of $b(j)$
far from $(E_L, E_R)$. The numbers $M$ and $A_n$ appear in the definitions for a technical reason.
\item If all the $b(j)$ are in $(E_L - 2 M - A_n, E_R + 2M + A_n)$, the product is empty,
and we set $D_n = 1$. When none of the $b(j)$ are in the interval ($\# I_n = 0$),
and actually in the more general case
\[ \liminf_{n \to \infty} \frac{\# I_n}{n} = 0~, \]
we show that
\[ \left| \Sigma_\text{ac}(J(a, b)) \cap (E_L, E_R) \right| = 0~.\]
\item Taking $E_L = -\infty$ and $E_R = +\infty$ in the theorem, we recover
(\ref{eq:rp.glob}) (since $D_n \geq A_n^{n - \# I_n}$.)
\end{enumerate}
\end{rmks*}

As an example, consider a periodic Jacobi matrix $J(1, b)$, where $b$ takes only two values
$0$ ($m$ times), and $R \geq 5$ ($\ell$ times.)  Then
\[ \Sigma_\text{ac}(J(1,b)) \subset (-2, 2) \cup (R-2, R+2)~, \]
since $J(1, b)$ can be considered as a perturbation of the diagonal matrix $J(0, b)$
by the free Laplacian $J(1, 0)$, $\| J(1, 0) \| = 2$.

Let us apply Theorem~\ref{thm:1.1} to every one of these two intervals. We have:
\[\begin{split}
|\Sigma_\text{ac}(J(1, b)) \cap (-2, 2)| &\leq \frac{4}{(R-4)^{\ell/m}}~, \\
|\Sigma_\text{ac}(J(1, b)) \cap (R-2, R+2)| &\leq \frac{4}{(R-4)^{m/\ell}}~.
\end{split}\]
When $R \to \infty$, both expressions tend to zero, thus the measure of the absolutely continuous
spectrum tends to zero.

\vspace{2mm}\noindent
Second, an even more elementary approach allows us to prove (directly) the following special case of (\ref{eq:ds.loc}).
Assume that $a, b$ are periodic sequences of period $q$ (namely, $a(n+q) = a(n), b(n+q) = b(n)$ for $n = 1,2,3,\cdots$.)
From the Bloch--Floquet theory (see, e.g., \cite{L}) $\Sigma_\text{ac}(J)$ is the union of $q$
closed intervals (bands), which may overlap only at the edges. Denote these bands $B_1, B_2, \cdots, B_q$
(ordered from right to left.)

\begin{thm}\label{thm:per} Under the assumptions above,
\begin{equation}\label{eq:bandlengths}
\left| B_j \right| \leq 2 A_q \left[ \cos \frac{\pi (j-1)}{q}  - \cos \frac{\pi j}{q} \right]
\end{equation}
for $j = 1, 2, \cdots, q$. Equality is attained if (and only if) $a$ and $b$ are constant.
\end{thm}

We emphasise again that Theorem~\ref{thm:per} is not new, and
follows in particular from (\ref{eq:ds.loc}). A parallel inequality
for periodic Schr\"{o}dinger operators on the real line can be
found, e.g., in the work of Garnett and Trubowitz \cite{GT}. We
provide a direct proof using extremal properties of polynomials, and
then use Theorem~\ref{thm:per} to recover the Deift--Simon
inequality (\ref{eq:ds.loc}) in full generality, and generalise it
to the non-ergodic case:

\begin{thm}\label{thm:ds.loc} Let $J(a, b)$ be a Jacobi operator, and let $n_i \uparrow \infty$
be a sequence such that the limit $k_{\{n_i\}} = \lim_{i \to \infty} k_{n_i}$ exists. Then
\begin{equation}\label{eq:loc1}
2 \pi \cdot \liminf_{i \to \infty} A_{n_i} \cdot \sin \left( \pi k_{\{n_i\}}(E) \right) \,
\frac{dk_{\{n_i\}}(E)}{dE} \geq 1
\end{equation}
for almost every $E \in \Sigma_\text{ac}(J)$.

\noindent In particular, if the integrated density of states exists for $J$, we have:
\begin{equation}\label{eq:loc}
2 \pi \cdot \lim_{n \to \infty} A_n
    \cdot \sin \left( \pi k(E) \right) \frac{dk(E)}{dE} \geq 1
\end{equation}
for almost every $E \in \Sigma_\text{ac}(J)$.
\end{thm}

Applying this result to an arbitrary partial limit of the sequence $k_n$ yields
another proof of the inequality (\ref{eq:rp.glob}).

\vspace{3mm}\noindent
{\bf Acknowledgment.} We are grateful to Yoram Last and to Jonathan Breuer
for the stimulating discussions, and for their helpful comments on the
preliminary version of this paper.

The first author is supported in part by The Israel Science Foundation (Grant No.\ 1169/06)
and by Grant 2006483 from the United States-Israel Binational Science Foundation (BSF),
Jerusalem, Israel. The second author is supported in part by the Adams Fellowship Program
of the Israel Academy of Sciences and Humanities and by the ISF.

\section{Preliminaries}\label{s:prel}

\subsection{Transfer Matrices}

Given an operator $J(a, b)$ of the form (\ref{eq:op}), we consider the associated
eigenvalue equation
\[ J \psi = E \psi~, \quad E \in \R~, \quad \psi: \N \to \C~.\]
For any $n \geq 1$, we consider the one-step transfer matrices
\[ \left( \begin{array}{cc} \frac{E-b(n)}{a(n)} & - \frac{a(n-1)}{a(n)} \\
                            1 & 0 \end{array} \right):
    \left( \begin{array}{c} \psi(n) \\ \psi(n-1) \end{array} \right)
    \mapsto \left( \begin{array}{c} \psi(n+1) \\ \psi(n) \end{array}\right)~,\]
and define the $n$-step transfer matrix
\begin{equation}\label{eq:phi}
\Phi_n(E) = \left( \begin{array}{cc} \frac{E-b(n)}{a(n)} & - \frac{a(n-1)}{a(n)} \\
                                    1                 & 0 \end{array} \right) \cdots
\left( \begin{array}{cc} \frac{E-b(1)}{a(1)} & - \frac{a(0)}{a(1)} \\
                                    1                 & 0 \end{array} \right)~.
\end{equation}

Denote
\begin{equation}\label{eq:a}
\mathcal{A}(J) = \left\{ E\in\R\big| \limsup_{n\rightarrow\infty}\frac{1}{n}\ln\|\Phi_n(E)\| = 0 \right\}~.
\end{equation}
Since
\[ \|\Phi_n\| \geq \sqrt{|\det{\Phi_n}|} = \frac{1}{\sqrt{|a(n)|}}~, \]
we have:
\[ \liminf \frac{1}{n}\ln\|\Phi_n(E)\| \geq \liminf \frac{1}{n} \ln \frac{1}{\sqrt{|a(n)|}}
    \geq - \limsup \frac{\ln |a(n)|}{2n} = 0~, \]
and therefore
\[
\frac{1}{n}\ln\|\Phi_n(E)\| \to 0 \quad \text{for every} \quad E \in \mathcal{A}~,
\]
and, since $|\tr \Phi_n| \leq 2 \| \Phi_n \|$,
\begin{equation}\label{eq:conv0}
\limsup_{n \to \infty} \frac{1}{n}\ln\left| \tr \Phi_n(E) \right| \leq 0 \quad \text{for every} \quad E \in \mathcal{A}~.
\end{equation}
Note that $\tr \Phi_n(E)$ is a real polynomial of $E$ with leading coefficient $A_n^{-n}$.
It follows from the Bloch--Floquet theory (see, e.g., \cite{L}) that
\begin{equation}\label{eq:bf}
\tr \Phi_n(E) = A_n^{-n} \det \left( E - J_n(a, b) \right)~,
\end{equation}
where $J_n = J_n(a, b)$ is the $n \times n$ matrix
\[ J_n = \left(\begin{array}{ccccccc}
    b(1) & a(1) & 0 & 0 & \cdots & 0 & -i a(n)  \\
    a(1) & b(2) & a(2) & 0 & \cdots & 0 & 0 \\
    0 & a(2) & b(3) & a(3) & \cdots & 0 & 0 \\
    0 & 0 & \ddots & \ddots & \cdots  & 0 & a(n-1) \\
    i a(n) & 0  & 0 & \cdots &  \cdots & a(n-1) & b(n)
    \end{array}\right)~.\]

Finally, from the subordinacy theory of Khan-Pearson \cite{KP},
\begin{equation}\label{eq:subord}
\Sigma_\text{ac}(J) \subset \mathcal{A}(J)~.
\end{equation}

\subsection{The Alternation Theorem, and a corollary}

Let $K\subset\R$ be a compact set. Denote
\begin{equation}\label{eq:Ln}
L_n(K) = \inf_{P_n\in\mathbb{P}_n} \max_{E\in{K}} |P_n(E)|~,
\end{equation}
where the infimum is over all monic polynomials of degree $n$.

\begin{thm*}[Chebyshev alternation theorem, see {\cite[\S I.5]{B}}] The infimum in $(\ref{eq:Ln})$
is attained on a unique polynomial $P_n$, which is uniquely characterized by the
following: there exists an $(n + 1)$-tuple of points in $K$
\[    E_1 > E_2 > E_3 > \cdots > E_{n+1} \]
on which $P_n$ attains the maximum with alternating signs:
\[ P_n(E_k) = (-1)^{k + 1} \max_{E \in K} |P_n(E)|~. \]
\end{thm*}

For example, $L_n([-2, 2]) = 2$, and the minimum is attained
for the scaled Chebyshev polynomials of the first kind:
\[ P_n(E) = 2 T_n(E/2), \quad T_n(\cos(\theta)) = \cos(n\theta)~. \]

\noindent We cite a corollary of the Chebyshev Alternation Theorem, due to P\'olya (see \cite{B}.)

\begin{prop}[P\'olya]\label{prop:mescap} If $K$ is a compact set,
\[ L_n(K) \geq \frac{|K|^n}{2^{2n - 1}}~. \]
\end{prop}

\noindent If $K$ is an interval, equality is achieved.

\section{Proof of Theorem~\ref{thm:1.1}}

According to (\ref{eq:conv0}) and (\ref{eq:subord}), the polynomials $\Delta_n = \tr \Phi_n$
satisfy
\begin{equation}\label{eq:subord.1}
\limsup_{n \to \infty} \frac{1}{n} \ln |\Delta_n(E)| \leq 0 \quad \text{for a.e.\ $E \in \Sigma_\text{ac}(J(a,b))$}~.
\end{equation}
Fix $\epsilon > 0$; by Egoroff's theorem there exists $\Sigma^\epsilon$ such that
\[ |\Sigma_\text{ac}(J(a,b)) \setminus \Sigma^\epsilon| < \epsilon \]
and the convergence in (\ref{eq:subord.1}) is uniform on
$\Sigma^\epsilon$. That is,
\[ \frac{1}{n} \ln |\Delta_n(E)| \leq \epsilon_n~, \quad E \in \Sigma^\epsilon~,\]
where $\epsilon_n \to 0$, and thus
\[ |\Delta_n(E)| \leq \exp \left( n \epsilon_n \right)~. \]
Now recall (\ref{eq:bf}) and consider the matrix $J_n(a, b)$ as the perturbation of $J_n(0, b)$
by the matrix $J_n(a, 0)$ of norm $\| J_n(a, 0) \| \leq 2 M$. We see that the zeros
$E_1, \cdots, E_n$ of $\Delta_n$ can be numbered so that $|E_j - b(j)| \leq 2M$.

Then, for $E \in (E_L, E_R)$,
\[\begin{split}
|\Delta_n(E)|
    &= A_n^{-n}
         \prod_{j \in I_n} |E - E_j| \, \prod_{j \notin I_n} |E - E_j| \\
    &\geq A_n^{-n} \,
         \prod_{j \in I_n} |E - E_j| \,\,\, D_n~.
\end{split}\]
Therefore
\[ L_{\# I_n}(\Sigma^\epsilon \cap (E_L, E_R))
    \leq \frac{\exp \left( n \epsilon_n \right) A_n^n}{D_n}~. \]
By P\'olya's theorem (Proposition~\ref{prop:mescap}),
\[ |\Sigma^\epsilon \cap (E_L, E_R)|
    \leq 4 \left[ \frac{\exp \left( n \epsilon_n \right) A_n^n}{D_n}\right]^{\frac{1}{\# I_n}}~.\]
If $\{\#I_n/n\}$ is bounded away from zero, we can conclude the proof, taking the lower limit
as $n \to \infty$ and then the limit as $\epsilon \to 0$.

Suppose $\liminf \#I_n/n = 0$. Then we prove a stronger statement:
\[|\Sigma_\text{ac}(J(a,b)) \cap (E_L, E_R)| = 0~. \]
For simplicity of notation we assume that $I_n$ is not empty for sufficiently
large $n$ (otherwise $\sigma(J(a, b)) \cap (E_L, E_R)$ is a finite set by the same perturbation arguments
as above.) Let $\delta_n > 0$ be a small
parameter that we shall choose later. If
\[ E \in (E_L + \delta_n, E_R - \delta_n) \cap \Sigma^\epsilon~, \]
then by definition of $I_n$
\[ \prod_{j \notin I_n} |E - E_j| \geq (A_n + \delta_n)^{n - \# I_n}, \]
therefore
\[ \prod_{j \in I_n} |E - E_j| \leq \frac{\exp \left( n \epsilon_n \right) A_n^n}
                                         {(A_n + \delta_n)^{n - \# I_n}}~.\]
Hence
\[ L_{\# I_n}(\Sigma^\epsilon \cap (E_L+\delta_n, E_R-\delta_n))
    \leq \frac{\exp \left( n \epsilon_n \right)A_n^n}{(A_n + \delta_n)^{n - \# I_n}}~,\]
and
\begin{multline*}
|\Sigma^\epsilon \cap (E_L+\delta_n, E_R-\delta_n)| \\
    \leq 4 \left[ {\frac{\exp \left( n \epsilon_n \right) A_n^n}
                              {(A_n + \delta_n)^{n - \# I_n}}}\right]^{\frac{1}{\# I_n}}
    = 4 A_n \left[ {\frac{\exp \left( n \epsilon_n \right)}
                              {(1 + \delta_n/A_n)^{n - \# I_n}}}\right]^{\frac{1}{\# I_n}}~.
\end{multline*}
Now we can choose $\delta_n = A_n \Big( \epsilon_n + \sqrt\frac{\#I_n}{n} \Big)$ and
take the lower limit as $n \to \infty$.

\section{Proof of Theorem~\ref{thm:per}}

Let $J(a, b)$ be a periodic Jacobi operator of period $q$. From the Bloch--Floquet
theory (see, e.g., \cite{L})
\[ \Sigma_\text{ac}(J) = \left\{ E \, \Big| \, |\Delta(E)| \leq 2 \right\} = \bigcup_{j=1}^q B_j~, \]
where $B_j$ are closed intervals (bands) that may overlap only at edges, and the
discriminant $\Delta = \Delta_q$ is a real polynomial of degree $q$ with leading coefficient
$LC(\Delta) = A_q^{-q}$. If two bands overlap at a point $E$, then $E$ is an edge, that is, $\Delta(E)=\pm2$.

Therefore there exist
\[ E_1 > E_2 > \cdots > E_{q-1} \]
so that $\Delta(E_j) = 2 \cdot (-1)^j$ (which are endpoints of the bands.) Vice versa,
if $\Delta$ is a polynomial of degree $q$ with positive leading coefficient for which such
points exist, the set
\[ \Big\{ E \, \big| \, |\Delta(E)| \leq 2 \Big\} \]
is the union of $q$ bands. Therefore we study the dependence of the lengths of the bands
on $E_1,\cdots, E_{q-1}$. The correspondence between $\Delta$ and $E_1, \cdots, E_{q-1}$
is not one-to-one, and we shall deal with this problem later.

Our main technical tool is the following general formula. Fix an $m$-tuple of points
$E_1, \cdots, E_m$, and let $s$ be a function that is non-zero and differentiable in the
neighbourhood of the points $E_i$; denote
\begin{equation}\label{eq:interp}
\mathcal{T}(E; E_1, \cdots, E_m)
    = \sum_{i=1}^m \frac{(-1)^i}{s(E_i)} \prod_{j \neq i} \frac{E - E_j}{E_i - E_j}
        + \prod_{i=1}^m (E - E_i)~,
\end{equation}
and
\[ \mathcal{B}_i(E) = \prod_{j \neq i}(E - E_j), \quad \mathcal{B}_i = \mathcal{B}_i(E_i)~. \]
The polynomial $\mathcal{T}(E) = \mathcal{T}(E; E_1, \cdots, E_m)$ is uniquely determined by the conditions
\[ \begin{cases}
\deg \mathcal{T} = m~, &\\
\mathcal{T}(E_i) = (-1)^i/s(E_i)~, &1 \leq i \leq m~, \\
\LC(\mathcal{T}) = 1
\end{cases} \]
(where again $\LC(P)$ stands for the leading coefficient of a polynomial $P$.)
\begin{prop}\label{prop:formula} For any $E^*, E_1, \cdots, E_{m}$
\[  \frac{\6}{\6 E_k} \mathcal{T}(E^\ast; E_1, \cdots, E_m)
    = -\frac{\mathcal{B}_k(E^\ast)}{\mathcal{B}_k s(E_k)}\frac{\6}{\6 E} \Big|_{E=E_k} \Big[ \mathcal{T}(E; E_1, \cdots, E_m) s(E)\Big]~.\]
\end{prop}
The proof of the proposition is via straightforward differentiation.\footnote{Similar methods were
used, for example, by Peherstorfer and Schiefermayr \cite{PS}. We thank Barry Simon for the reference.}
\begin{proof}[Proof of Theorem~\ref{thm:per}]
Fix $1 \leq j \leq q$, and let us show that
\[ \left| B_j \right| \leq 2 A_q \left[ \cos \frac{\pi (j-1)}{q}  - \cos \frac{\pi j}{q} \right]~.\]
There is a point $E_0 \in B_j$ such that $\Delta(E_0) = 0$, and without loss of generality $E_0 = 0$ (else replace
$b$ with $b - E_0$.) Then
\[ \Delta(E) = (a(1) \cdots a(q))^{-1} E \,\mathcal{T}(E)~, \]
where $\mathcal{T}(E)$ is a polynomial of degree $m = q-1$ such that
\[ \LC(\mathcal{T}) = 1, \quad \mathcal{T}(E_i) = (-1)^i \, \frac{2 \cdot a(1) \cdots a(q)}{E_i}~. \]
Therefore $\mathcal{T}$ is given by (\ref{eq:interp}) with $s(E) = (2 a(1) \cdots a(q))^{-1} \, E$. Fix $E^* \in B_j$,
then
\[ E_1 > E_2 > \cdots > E_{j-1} \geq E^* \geq E_j > \cdots > E_{q-1}~. \]
It is easy to see that the discriminant of the free Laplacian is given by
\[ \Delta_{J(1,0)} (E) = 2 T_q(E/2)~, \]
where $T_q$ is the $q$-th Chebyshev polynomial of the first kind. Therefore
\[ B_j(J(1, 0)) = [ 2 \cos \frac{\pi j}{q}, \, 2 \cos \frac{\pi(j-1)}{q}]~. \]
We shall show that $E^*$ also lies in the $j$-th band of
\begin{equation}\label{eq:lapl.bu}
J \left( A_q, - 2 \cos \frac{\pi(j - 1/2)}{q} \right)~,
\end{equation}
which is the free Laplacian, viewed as a periodic operator of period $q$, and shifted so
that $0$ is the $j$-th zero of its discriminant. The theorem immediately follows.

Without loss of generality assume $E^* > 0$. Fix $1 \leq k \leq q-1$, and let us study how $\Delta(E^*)$
varies with the change of $E_k$. We apply Proposition~\ref{prop:formula}, assuming for now that $E^* \neq E_{j-1}$.
It is easy to see that
\begin{equation}\label{sign1}
\sign \frac{\mathcal{B}_k(E^*)}{\mathcal{B}_k s(E_k)}  = (-1)^{j+k+1}~.
\end{equation}
Next, $\mathcal{T}(E) s(E)$ assumes the value $(-1)^k$ at two points in $(E_{k-1}, E_{k+1})$ which we denote $E_k^- \leq E_k^+$
($E_k^+$ is the left edge of $B_k$, and $E_k^-$ is the right edge of $B_{k+1}$.) Note that
\begin{multline*}
\mathcal{T}(E; E_1, \cdots, E_{k-1}, E_k^+, E_{k+1}, \cdots, E_{q-1}) \\
= \mathcal{T}(E; E_1, \cdots, E_{k-1}, E_k^-, E_{k+1}, \cdots, E_{q-1})~,
\end{multline*}
that is, the correspondence between $\mathcal{T}$ and the points $E_i$ is not one-to-one, and we shall
use this shortly.

One can see from Figure~1 that $\mathcal{T}(E)s(E)$ is increasing at $E = E_k^+$ if and only if $k$ is odd, and the opposite for
$E_k^-$:
\begin{equation}\label{sign2}
\sign \frac{d}{dE} \Big|_{E = E_k^\pm} \mathcal{T}(E) s(E) = \mp (-1)^k~,
\end{equation}
and also that
\begin{equation}\label{sign3}
\sign \mathcal{T}(E^*) = \sign \Big[ \mathcal{T}(E^*) s(E^*) \Big] = (-1)^{j+1}~.
\end{equation}

\begin{figure}[h]
\vspace{2.5cm}
\setlength{\unitlength}{1cm}
\begin{pspicture}(-2,-2)
\psline(0,0)(10,0)
\psset{arrows=-}
\psplot[plotstyle=curve]{.6}{9}{ x 1 sub x 3 sub mul x 6.5 sub mul x 8.7 sub mul 20 div}
\psline(0,1)(10,1)
\psline(0,-1)(10,-1)
\psline[linewidth=.09](0.75, 0)(1.4, 0)
\rput(1.4, .3){$E_3^-$}
\psline(1.4,.1)(1.4,-.1)
\rput(2.42,.3){{$E_3^+$}}
\psline(2.42,.1)(2.42,-.1)
\rput(2.95, .35){$0$}
\psline(3,.1)(3,-.1)
\rput(3.2, -.4){$E^*$}
\psline(3.2,.1)(3.2,-.1)
\rput(3.52,.3){{$E_2^-$}}
\psline(3.52,.1)(3.52,-.1)
\psline[linewidth=.09](2.42, 0)(3.52, 0)
\rput(6, .3){$E_2^+$}
\psline(6,.1)(6,-.1)
\rput(7,.3){$E_1^-$}
\psline(7,.1)(7,-.1)
\psline[linewidth=.09](6, 0)(7, 0)
\rput(8.42, .3){$E_1^+$}
\psline(8.42,.1)(8.42,-.1)
\psline[linewidth=.09](8.42, 0)(8.9, 0)
\rput(-.5,1){$+1$}
\rput(-.5,-1){$-1$}
\end{pspicture}
\caption[Fig. 1]{$\mathcal{T}(E)s(E)$ with $q = 4$ and $j = 3$}\label{fig:1}
\end{figure}
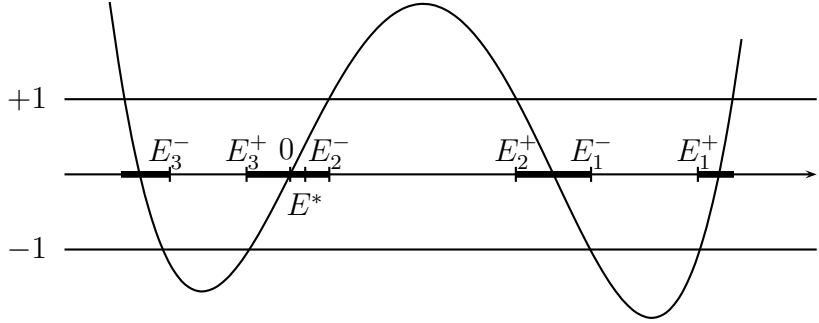

\noindent
Combining (\ref{sign1}), (\ref{sign2}), (\ref{sign3}), and Proposition~\ref{prop:formula},
we obtain:
\[
\sign \left[ \frac{\6}{\6 E_k} \left| \mathcal{T}(E^*) \right| \right] =
\begin{cases}
    -1, &E_k = E_k^-\\
    +1, &E_k = E_k^+
\end{cases}~.\]
Therefore, the value of $\left| \mathcal{T}(E^*) \right|$ decreases if and only if $E_k^-$ moves to the right,
which happens if and only if $E_k^+$ moves to the left. That is, if $E_k^-$ moves to the right, $E_k^+$ moves to the left, until they coincide.
Thus $|\mathcal{T}(E^*)|$ is minimal if $E_k^- = E_k^+$ (when the $k$-th band is glued to the $(k+1)$-th.)

This is true for any
$k = 1, 2, \cdots, q-1$, therefore $|\mathcal{T}(E^*)|$ (and $|\Delta(E^*)|$) is minimal when all the bands are glued together.
According to the Chebyshev Alternation Theorem, this is the case if and only if
\[ \Delta(E) = 2 T_q \left( \frac{E - 2 \cos \frac{\pi(j - 1/2)}{q}}{2 A_q} \right)~,\]
which is exactly the discriminant of (\ref{eq:lapl.bu}).

\end{proof}

\section{Proof of Theorem~\ref{thm:ds.loc}.}

We first consider the {\bf periodic case}, which is somewhat technically simpler,
and will also be used in the proof of the general case.

Assume that $J(a, b)$ is periodic of period $q$, with discriminant $\Delta_q$, and bands
\[ B_1 = [\ell_1, r_1], \cdots, B_q = [\ell_q, r_q]~.\]
We recall that a periodic operator can be considered ergodic (with respect to the ergodic system
$(\Z/q\Z, q^{-1}\sum_{i=0}^{q-1} \delta_{i}, \cdot \mapsto \cdot + 1)$, hence its density
of states is well defined. Let us discuss its structure.

The measure $k_q$ has one atom of mass $1/q$ in (the interior of) every $B_j$, hence
$k_q(r_j) - k_q(\ell_j) = 1/q$. Now, if we consider $J(a, b)$ as a periodic operator of period
$nq$, $n \geq 1$, then
\begin{equation}\label{eq:doubl}
\Delta_{nq}(E) = 2 T_n(\Delta_q(E)/2)~.
\end{equation}
Indeed, both sides of (\ref{eq:doubl}) are polynomials of degree $nq$ with leading coefficient
\[ A_{nq}^{-nq} = \big[ A_q^{-q} \big]^n \]
and with maximal absolute value $2$ on the spectrum attained $nq+1$
times with alternating signs, hence they coincide according to Chebyshev's
alternation theorem.

Therefore every band $B_j$ splits exactly into $n$ bands, and
\[ k_{nq}(r_j) - k_{nq}(\ell_j) = \frac{n}{nq} = \frac{1}{q}~. \]
Passing to the limit as $n \to \infty$, we obtain
\[ k(r_j) - k(\ell_j) = \frac{1}{q}~.\]
The function $k^{-1}$ is well-defined outside a countable set, and
\[ k^{-1} \left( \frac{j-1}{q} + 0 \right) = \ell_j~,
   k^{-1} \left( \frac{j}{q} - 0 \right) = r_j~.\]
According to Theorem~\ref{thm:per},
\[\begin{split}
&k^{-1} \left( \frac{j}{q} - 0 \right) - k^{-1} \left( \frac{j-1}{q} + 0 \right) \\
    &\qquad= |B_j| \leq  2 A_q \left[ \cos \frac{\pi (j-1)}{q}  - \cos \frac{\pi j}{q} \right] \\
    &\qquad= \int_{\frac{j-1}{q}}^{\frac{j}{q}}  2\pi A_q \, \sin \pi x \, dx~,
\end{split}\]
and also
\begin{multline}\label{eq:bu}
k^{-1} \left( \frac{j}{nq} - 0 \right) - k^{-1} \left( \frac{j-1}{nq} + 0 \right) \\
    \leq \int_{\frac{j-1}{nq}}^{\frac{j}{nq}}  2\pi A_q \, \sin \pi x \, dx~,
    \quad n = 1,2,\cdots~.
\end{multline}
Passing to the limit as $n \to \infty$, we obtain:
\[ \frac{d}{dx} k^{-1}(x) \leq 2\pi A_q \, \sin \pi x \quad \text{a.e. on $[0, 1]$.}\]
Taking $x = k(E)$ and using the chain rule, we obtain (\ref{eq:loc}).

\vspace{3mm}\noindent In the {\bf general case}, fix $q \geq 1$, and choose $q-1$
points $p_1 > p_2 > \cdots > p_{q-1}$
so that $k_{\{n_i\}}(p_j) = (q-j)/q$. Denote $I_j = [p_{j}, p_{j-1}]$ (with
$p_0 = +\infty$, $p_q = -\infty$), and set $A_- = \liminf_{i \to \infty} A_{n_i}$.

\begin{lemma}\label{lem}
$\left| I_j \cap \Sigma_\text{ac}(J(a,b)) \right|
 \leq 2 A_- \left[ \cos \frac{\pi (j-1)}{q}  - \cos \frac{\pi j}{q} \right]$.
\end{lemma}

The lemma implies the theorem. Indeed, let
\[ m(E) = \left|\Sigma_\text{ac}(J(a,b)) \cap (-\infty, E]\right|~. \]
We have:
\[ m(p_{j}) - m(p_{j-1}) \leq 2 A_- \left[ \cos \frac{\pi (j-1)}{q}  - \cos \frac{\pi j}{q} \right]~; \]
applying this with every $q$ and $1 \leq j \leq q-1$, we obtain
\[ \frac{dm(E)}{dE} \leq 2 \pi A_- \sin \left( \pi k_{\{n_i\}}(E) \right) \frac{dk_{\{n_i\}}(E)}{dE} \]
for a.e.\ $E \in \Sigma_\text{ac}(J(a,b))$.  According to the Lebesgue theorem,
almost every $E \in \Sigma_\text{ac}(J(a, b))$ is a Lebesgue point, that is,
$ \frac{dm(E)}{dE}=1$ for a.e.\ $E \in \Sigma_\text{ac}(J(a,b))$. Therefore the
desired inequality follows.

\begin{proof}[Proof of Lemma~\ref{lem}]
Consider the polynomials $\Delta_{n}$. Passing to a subsequence, we can assume
that
\begin{equation}\label{eq:subseq}
\lim_{i \to \infty} \LC(\Delta_{n_i})^{1/n_i} = 1 / A_-~.
\end{equation}

\noindent For any $\epsilon > 0$ one can find a set $S_i^0$ such that
\[ \left| \Sigma_\text{ac}(J(a,b)) \setminus S_i^0 \right| \leq \epsilon~, \]
and
\[ \lim_{i \to \infty} \max_{E \in S_i^0} |\Delta_{n_i}|^{1/n_i} \leq 1~. \]
This follows from the Khan--Pearson theorem combined with Egoroff's theorem,
as in the proof of Theorem~\ref{thm:1.1}.

We choose $\delta_i \to 0$ so that
\[ \max_{E \in S_i^0} |\Delta_{n_i}|^{1/n_i} \leq 1 + \delta_i~, \]
and let
\begin{equation}\label{eq:si+}
S_i^+ = \left\{ E \, \mid |\Delta_{n_i}|^{1/n_i} \leq 1 + \delta_i \right\}~.
\end{equation}
The polynomial $\Delta_{n_i}$ coincides with the discriminant $\Delta$ of a periodic
operator of period $n_i$ obtained by periodising the first $n_i$ values of $a, b$. Therefore
the smaller set $S_i = \left\{ |\Delta_{n_i}| \leq 2 \right\} \subset S_i^+$ (see Figure~2) consists of $n_i$
bands, and we have:
\[ \left| I_j \cap S_i \right|
    \leq  2 A_- (1 + o(1)) \left[ \cos \frac{\pi (j-1)}{q}  - \cos \frac{\pi j}{q} \right]\]
by (\ref{eq:subseq}) and the periodic case that we have already considered. Let us show that the last
inequality holds also for $|I_j \cap S_i^+|$.

\begin{figure}[h]
\vspace{2.5cm}
\setlength{\unitlength}{1cm}
\begin{pspicture}(-2,-2)
\psline(0,0)(10,0)
\psset{arrows=-}
\psplot[plotstyle=curve]{.6}{9.1}{ x 1 sub x 3 sub mul x 6.5 sub mul x 8.7 sub mul 20 div}
\psline(0,1)(10,1)
\psline(0,-1)(10,-1)
\psline[linestyle=dotted](0,1.8)(10,1.8)
\psline[linestyle=dotted](0,-1.8)(10,-1.8)
\psline(0.65, .2)(3.98, .2)
\psline(5.5, .2)(7.5,.2)
\psline(8.1, .2)(9, .2)
\psline[linewidth=.09](0.75, 0)(1.4, 0)
\psline(.75, .1)(.75, -.1)
\psline(1.4,.1)(1.4,-.1)
\psline(2.42,.1)(2.42,-.1)
\psline(3.52,.1)(3.52,-.1)
\psline[linewidth=.09](2.42, 0)(3.52, 0)
\psline(6,.1)(6,-.1)
\psline(7,.1)(7,-.1)
\psline[linewidth=.09](6, 0)(7, 0)
\psline(8.42,.1)(8.42,-.1)
\psline[linewidth=.09](8.42, 0)(8.9, 0)
\psline(8.9,.1)(8.9, -.1)
\rput(-.5,1){$+2$}
\rput(-.5,-1){$-2$}
\rput(-1,1.8){$(1+\delta_i)^{n_i q}$}
\rput(-1,-1.8){$-(1+\delta_i)^{n_i q}$}
\end{pspicture}
\caption[Fig. 1]{$S_i$ (bold) and $S_i^+$ (drawn slightly above)}\label{fig:2}
\end{figure}
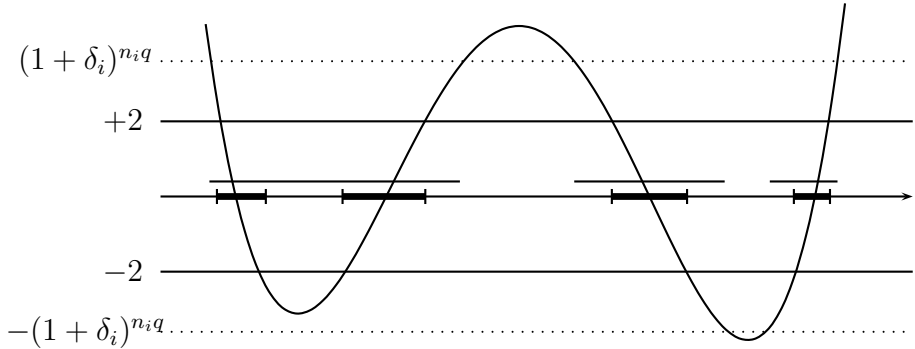

Applying Proposition~\ref{prop:formula} similarly to the proof of Theorem~\ref{thm:per},
we see that the measure $\left| I_j \cap S_i^+ \right|$ is maximal if $S_i^+$ is an interval
(if it has gaps, the length increases if one closes them.) Let us study this extremal case.

The interval $S_i^+$ contains a subset $S_i$ with $L_{n_i}(S_i)^{1/n_i} \to A_-$, and on the other
hand $L_{n_i}(S_i^+)^{1/n_i} \leq A_-(1 + o(1))$ by construction (\ref{eq:si+}). Therefore
$L_{n_i}(S_i^+)^{1/n_i} = A_- (1 + o(1))$. Since $S_i^+$ is an interval, and
$L_n(I) = |I|^n / 2^{2n-1}$ for any interval $I$ and any $n \in \N$, we deduce that
$ |S_i^+| = 4 A_- (1 + o(1)) $. Therefore the polynomials $A_{n_i}^{n_i} \Delta_{n_i}$
are asymptotically extremal in the definition of $L_{n_i}$, namely:
\begin{equation}\label{eq:asympextr}
\lim_{i \to \infty} \max_{E \in S_i^+} |A_{n_i}^{n_i} \Delta_{n_i}(E)|^{1/n_i} =
    \lim_{i \to \infty} L_{n_i}(S_i^+)^{1/n_i}~.
\end{equation}

Now we appeal to Szeg\H{o}'s theorem \cite{Sz} (see Blatt and Saff \cite{BlS}
for a more recent discussion of this result, as well as numerous generalisations.) It states
that, under the assumption (\ref{eq:asympextr}), the distribution of the zeros of
$A_{n_i}^{n_i} \Delta_{n_i}$ is asymptotically the same as that of the extreme polynomials
$P_{n_i}$ in the definition of $L_{n_i}$, which are the (properly scaled) Chebyshev polynomials
of the first kind.

Recall that, according to the definition of the points $p_j$, the fraction of zeros of
$\Delta_{n_i}$ that fall into $I_j$ is asymptotically $1/q$, and,
similarly, the fractions that fall into the two half-lines of its complement are asymptotically
$(j-1)/q$ and $(q-j)/q$ (respectively). Therefore the same holds for $P_{n_i}$, and hence
\[ \lim_{i \to \infty} |I_j \cap S_i^+|
    = 2 A_- \left[ \cos \frac{\pi (j-1)}{q}  - \cos \frac{\pi j}{q} \right]~.\]

We have derived the last equality for the case when the left-hand side is maximal, hence
in the general case we have the inequality
\[ \lim_{i \to \infty} |I_j \cap S_i^+|
    \leq 2 A_- \left[ \cos \frac{\pi (j-1)}{q}  - \cos \frac{\pi j}{q} \right] \]
that was claimed. Letting $\epsilon \to 0$ we conclude the
proof of the lemma.

\end{proof}

\end{document}